\documentclass[a4paper,11pt]{article}
\usepackage{Lutsko_Style_Balint}
\usepackage[T1]{fontenc}
\usepackage[title]{appendix}

\newcommand{\vep}{\varepsilon}

\begin{document}

\title{Diffusion of the Random Lorentz Process in a Magnetic Field}

\author{
{\sc Christopher Lutsko$^1$ and B\'alint T\'oth$^{2,3}$}
\\[8pt]
$^1$ University of Houston, USA
\\
$^2$ University of Bristol, UK
\\
$^3$R\'enyi Institute, Budapest, HU
}

\maketitle

\begin{abstract}
  \noindent Consider the motion of a charged, point particle moving in the complement of a Poisson distribution of hard sphere scatterers in two dimensions under the effect of a fixed magnetic field. Building on, and extending a coupling method established by the authors, \cite{lutsko-toth_19}, we show that this 'magnetic Lorentz gas' satisfies an invariance principle in an intermediate scaling limit. That is, we apply the low-density (Boltzmann-Grad) limit and simultaneously take the limit as time goes to infinity, then prove convergence of the rescaled trajectory to a Brownian motion in this limit.

\medskip\noindent
{\sc MSC2010: } 60F17; 60K35; 60K37; 60K40; 82C22; 82C31; 82C40; 82C41

\medskip\noindent
{\sc Key words and phrases:} Lorentz gas; low-density limit; invariance principle.

\end{abstract}


\section{Introduction}
\label{s:Introduction}

We study the long time behavior of the random, magnetic Lorentz gas in two dimensions. That is, let $\cP$ be a Poisson point process in $\R^2$ of intensity $\varrho>0$. Now place a spherical scatterer of radius $\varepsilon>0$ and infinite mass, centered at each point in $\cP$ and fix  a constant magnetic field of intensity $B$ perpendicular to the plane. With that, we consider the motion of a charged, point particle moving with unit speed in the complement of the scatterer set and colliding elastically with the fixed scatterers. We call this the \emph{magnetic Lorentz process} (abbreviated MLP). We illustrate some sample trajectories for this model in Figure \ref{fig:trajectory}.

This model presents a simple, yet highly non-trivial example of particle dynamics and has been very well studied. For instance, in a recent paper, \cite{nota-saffirio-simonella_19} it is proved, using methods inspired by \cite{gallavotti_69}, that in the low density, or Boltzmann-Grad limit, and annealed setting, the MLP converges in distribution, in any compact microscopic time interval, to the limiting process corresponding to the linear magnetic Boltzmann equation derived and analysed in \cite{bobylev_95, bobylev_97, bobylev_01, kuzmany-spohn_98}.

The limiting velocity process (described in detail in section \ref{ss: The limit  process}) is essentially Markovian and sufficiently fast mixing. Thus, it easily follows that if we \emph{subsequently apply a diffusive limit} we obtain the invariance principle for this limit process. However, we stress that this drops out by \emph{first taking a kinetic limit}, and \emph{then applying a diffusive limit}.

In the present paper, we present an alternative, purely probabilistic  approach to studying this model. This approach is in the spirit of our earlier papers \cite{lutsko-toth_19, lutsko-toth_20} where the scaling limit of the Lorentz gas in $d\ge3$ and, respectively, of the Ehrenfest wind-tree model in $d\ge2$ were studied. Not only does our approach recover the result of \cite{nota-saffirio-simonella_19} for finite microscopic times, it allows us to consider the \emph{joint kinetic and diffusive limit} for time-scales going to infinity. That is, we consider the displacement of the MLP in the limit as we apply Boltzmann-Grad scaling, and simultaneously diffusive time scaling. For an explicit formulation see Theorem \ref{thm:main-theorem}.

The rest of the paper is structured as follows. In section \ref{s: The processes} we formally define the processes involved: the physical MLP, the low-density (Boltzmann-Grad) limit process, and an auxiliary 'Markovized' version of the MLP. We also state our main result and outline our proof in this section. The proof is given in section \ref{s: Proof}, partitioned into three subsections. Section \ref{s: Appendix: Random walk estimates} is an Appendix in which we collected ingredients pertaining to our random walk and Green's functions.


\begin{figure}[t!]
  \begin{center}    
    \includegraphics[width=0.4\textwidth]{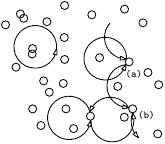}
  \end{center}
  \caption{%
    {\tt Here we give two examples of the magnetic Lorentz gas travelling through the same array of scatterers. Collision (a) is an example of a recollision with the 'current scatterer', while collision (b) is an example of a recollision with a 'past scatterer'.}   
  }%
  
  \label{fig:trajectory}

\end{figure}

\section{The processes}
\label{s: The processes}

\subsection{The magnetic Lorentz process}
\label{ss:The magnetic Lorentz process}

First we formally define the physical MLP, and then state the main result of this paper. Given a Poisson point process $\cP$ of intensity $\varrho>0$, magnetic field of intensity $B>0$ perpendicular to the plane, radius $\varepsilon>0$, and initial velocity $U_0 \in \mathbb{S}^1$ on the unit circle, we denote the position at time $t>0$ of a magnetic Lorentz process starting at the origin: $X^{\varepsilon,\varrho}(t) \in\R^2$. We will always assume that the starting point is in the complement of the area covered by scatterers. The process $X^{\varepsilon,\varrho}$ moves in anti-clockwise circular orbits of radius: $R_L:= \frac{1}{B}$, the 'Larmor radius' (assuming the mass and charge of the particle are unit). For the sake of simplicity we assume $B=1$. 

The following scenarios of trajectory could occur with positive probability. For a detailed analysis of the various trapping scenarios see \cite{kuzmany-spohn_98}.
\begin{enumerate}[(A)]
\item
\label{perctrap}
The starting point is fully surrounded by a connected cluster of partially overlapping scatterers and thus the trajectory is trapped in a compact domain. Due to a simple percolation argument this event will happen with probability $\theta_1(\varrho\vareps^2)$, where the percolation function $\theta_1:\R_+\to\R_+$ is nondecreasing, and $\lim_{u\to0}\theta_1(u)=0$. 
\item 
\label{larmororbit}
The trajectory of the charged particle goes around on a full Larmor orbit without encountering a scatterer, and thus it is again trapped in a compact domain. This event happens with probability $e^{-4\pi \varrho\varepsilon}$ (if $\varepsilon\leq1$).
\item 
\label{magnettrap}
The trajectory of the charged particle encounters a finite positive number of scatterers, and densely fills the union of discs of radii $2+\vareps$ centered at these finitely many scatterers, which in this case must be a connected compact set of the plain. The probability of this event is bounded from above by $\theta_2(\varrho)$ where $\theta_2:\R_+\to\R_+$ is nonincreasing, and $\lim_{u\to \infty}\theta_2(u)=0$. 
\item 
\label{notrap}
None of the previous trapping scenarios occur and thus, the trajectory evolves unboundedly. 
\end{enumerate}
We will denote by $A$, $B$, $C$, and $D$ the events described above verbally.  Examples of the various scenarios are shown in Fig.  \ref{fig:scenarios}.


\begin{figure}[t!]
  \begin{center}    
    \includegraphics[width=0.9\textwidth]{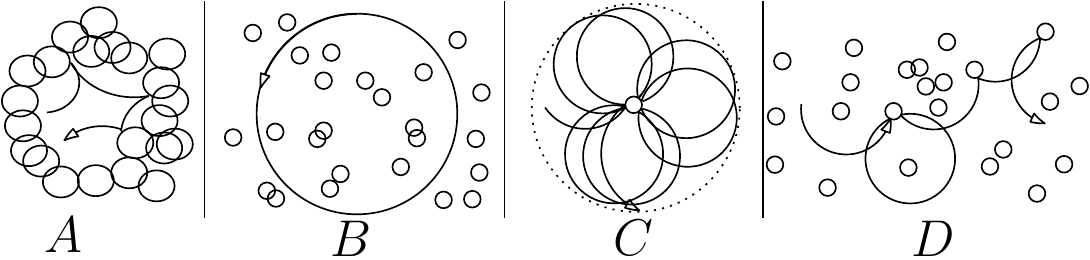}
  \end{center}
  \caption{%
    {\tt The four possible scenarios for the magnetic Lorentz gas are shown here as described above.}   
  }%
  
  \label{fig:scenarios}

\end{figure}

A major open problem in this context is to understand the long time diffusive behaviour of the trajectory of the charged particle in the above setting.  Assume that the parameters $\varepsilon$ and $\varrho$ are chosen so that the probability of scenario \ref{notrap} is positive. The Holy Grail would be to prove the Invariance Principle (a.k.a. functional Central Limit Theorem) for the diffusively scaled trajectory
\begin{align}
\label{Diff}
\frac{X^{\varepsilon,\varrho}(tT)}{\sqrt{T}}  
\qquad 
\mbox{ as } 
\qquad
T \to \infty, 
\end{align}
conditionally on the event $D$. With our present knowledge and technical methods, this goal is, unfortunately, beyond reach. 

The setting of the cited works and of the present note is the \emph{Boltzmann-Grad (low density)} limit: 
\begin{align}
\label{BG-Limit}
\varrho\to\infty, 
\qquad
\varepsilon \to 0, 
\qquad 
\varrho \varepsilon \to 2.  
\end{align}
In this limit, the mean free flight time will have length $1$, and the probabilities of the above listed scenarios converge as follows 
\begin{align}
\label{scenarioprobabs}
\probab{A}\to0, 
\qquad
\probab{B}\to e^{-8\pi}, 
\qquad
\probab{C}\to0, 
\qquad
\probab{D}\to 1-  e^{-8\pi}. 
\end{align}
Thus, in this limit there are two scenarios left: the trajectory either moves indefinitely on a circular Larmor trajectory, or, if it experiences one collision with a scatterer then it will collide with infinitely many distinct scatterers.

Our main result is 

\begin{theorem} 
\label{thm:main-theorem}
Let $T(\varepsilon)$ be such that 
$\lim_{\varepsilon\to 0} T(\varepsilon) =\infty $, 
$\lim_{\varepsilon\to 0} \varepsilon\abs{\log{\varepsilon}} T(\varepsilon)=0$. Then, 
under the Boltzmann-Grad limit \eqref{BG-Limit},
\begin{align}
\label{inv princ}
\big\{
t\mapsto 
T(\varepsilon)^{-1/2}
X^{\varrho, \varepsilon}(T(\varepsilon)t)
\big\}
\Rightarrow
\big\{
t\mapsto
\alpha W(t)
\big\}
\end{align} 
where $\Rightarrow$ denotes weak convergence of trajectories in $C([0,1]\to\R^2)$, and, on the right hand side, $\alpha$ is a Bernoulli random variable with distribution $\probab{\alpha=0}= e^{-8\pi}=1-\probab{\alpha=\sigma}$, $\sigma>0$ is a finite positive constant, and $t\mapsto W(t)$ is a standard Wiener process on $\R^2$, independent of the random variable $\alpha$. 
\end{theorem}

The result of the recent paper \cite{nota-saffirio-simonella_19} is to show that under the Boltzmann-Grad limit \eqref{BG-Limit},
\begin{align}
\label{nss-limit}
\big\{
t\mapsto 
X^{\varrho, \varepsilon}(t)
\big\}
\Rightarrow
\big\{
t\mapsto
Y(t)
\big\}
\end{align}
where $\Rightarrow$ denotes weak convergence of trajectories in $C([0,T]\to\R^2)$ on any \emph{fixed  compact time} interval $[0,T]$ and $t\mapsto Y(t)$ is the process explicitly constructed in section \ref{ss: The limit process}.

Since the limit process $t\mapsto Y(t)$ is constructed in a simple transparent way as an additive functional of a Markov chain with strong mixing properties -- see section \ref{ss: The limit process} -- from well established martingale approximation methods (see e.g. \cite{komorowski-landim-olla_2012}) it follows, that, in a second diffusive limit, we get
\begin{align}
\label{two-steps-inv-princ}
\big\{
t\mapsto 
T^{-1/2}
Y(Tt)
\big\}
\Rightarrow
\big\{
t\mapsto
\alpha W(t)
\big\}, 
\qquad
\text{ as \ }
T\to\infty, 
\end{align} 
where $\Rightarrow$ denotes weak convergence of trajectories in $C([0,1]\to \R^2)$. However we stress again that this drops out of taking a subsequent diffusive limit, not from applying the limits simultaneously as we do in Theorem \ref{thm:main-theorem}.

\noindent
\textbf{Remark:}
In Theorem \ref{thm:main-theorem} and in \eqref{nss-limit}, the limit is meant in the annealed setting. That is, the convergence in distribution is with respect to the Poisson distribution of the scatterers and the initial velocity. The methods employed in \cite{nota-saffirio-simonella_19} and in this note are not suited for handling the quenched setting.

\subsubsection*{Outline of the proof of Theorem \ref{thm:main-theorem}}
\label{sss:Outline of the proof of Theorem thm:main-theorem}

From now on we fix $\varrho = 2\varepsilon^{-1}$ and simplify our notation to $X^{\varepsilon,\varrho}(t)=:X^{\varepsilon}(t)$.

The proof will follow the outline of what we called "naive coupling" in \cite{lutsko-toth_19}. There are, however,  some significant differences in  the details which we will emphasize on the way. We stress that we do not go beyond the "naive coupling" since anyway the more sophisticated coupling of \cite{lutsko-toth_19} yields an improved result (longer time scale of the validity of the invariance principle) only in $d\geq 3$.

For $\varepsilon>0$, we will define in section \ref{ss: The markovized version of the MLP} a so-called Markovized version $t\mapsto Y^{\varepsilon}(t)$ of the MLP. This process is defined similarly to the physical MLP,  $t\mapsto X^{\varepsilon}(t)$, except that after colliding with a fresh scatterer, the Markovized process forgets the previous scatterers. Therefore, in Figure \ref{fig:trajectory} the recollisions with the 'current scatterer' (e.g scattering {\tt (a)}) are observed, while recollisions with 'past scatterers' (scattering {\tt (b)}) are ignored. Likewise, fresh scatterers may be encountered in domains where for the physical process the existence of a past scatterer is forbidden (we call these shadowed scatterings). For a formally precise construction of the Markovized process see section  \ref{ss: The markovized version of the MLP}.

Since the Markovized process $t\mapsto Y^{\varepsilon}(t)$ is constructed in a simple transparent way as an additive functional of a Markov chain with strong mixing properties, uniformly in $\varepsilon>0$ (see section \ref{ss: The markovized version of the MLP}), from well established martingale approximation methods (see e.g. \cite{komorowski-landim-olla_2012}) the following invariance principle holds: For any sequence $T(\varepsilon)\to\infty$
\begin{align}
\label{inv-princ-for-Markovized}
\big\{
t\mapsto 
T(\varepsilon)^{-1/2}
Y^{\varepsilon}(T(\varepsilon)t)
\big\}
\Rightarrow
\big\{
t\mapsto
\alpha W(t)
\big\}, 
\qquad
\text{ as \ }
T\to\infty, 
\end{align} 
where $\Rightarrow$, $\alpha$ and $W(t)$ are as in Theorem \ref{thm:main-theorem}. 

It remains to prove that the processes $t\mapsto X^{\varepsilon}(t)$ and $t\mapsto Y^{\varepsilon}(t)$ can be realized jointly (i.e. coupled) so that they stay equal with high probability, up to times of order $T(\varepsilon)=\ordo((\varepsilon\abs{\log\varepsilon})^{-1})$. This is the content of Theorem \ref{thm:XY} below, which is the main technical result of this note.  

\subsection{The limit process}
\label{ss: The limit process}

In this section we construct explicitly the limit process $t\mapsto Y(t)$ which appears in the low density (Boltzmann-Grad) kinetic limit \eqref{nss-limit}. We stress that there is no novelty here: the process had been introduced and analysed in \cite{bobylev_95, bobylev_97, bobylev_01, kuzmany-spohn_98}. However, our presentation may be  more transparent from a probabilistic point of view. 

We will denote by 
\begin{align*}
&
t\mapsto U(t)\in{\mathbb S}^1, 
&& 
t\mapsto Y(t):=\int_0^t U(s)\, ds
\end{align*}
the limiting velocity and position process. 
It will be convenient to use complex coordinates. We construct a \emph{piecewise linear, continuous from the left} process
\begin{align*}
t\mapsto \varphi(t)\in \R/(2\pi), 
\end{align*}
such that 
\begin{align}
\label{U and Y}
&
U(t)=e^{i\varphi(t)}, 
&&
Y(t)=\int_0^t e^{i\varphi(s)}\, ds. 
\end{align}
The ingredients of the construction are the following random variables, all fully independent between them. 
\begin{align*}
\varphi_0 
&
\sim 
{\tt UNI}[0, 2\pi],
&&
\frac{d}{dx} \probab{\varphi_0<x}
=
\frac{1}{2\pi}\one(0 \le x \le 2\pi),
\\
\xi_{j}
&
\sim
{\tt EXP}(1),
&&
\frac{d}{dx} \probab{\xi_j<x}
=
e^{-x}\one(0 \le x < \infty),
&&&
1\le j < \infty,
\\
\alpha_{j}
&
\sim
2\arccos(
{\tt UNI}[-1,1]),
&&
\frac{d}{dx} \probab{\alpha_j<x}
=
\frac{1}{4} \sin(\frac{x}{2})\one(0 \le x < 2\pi),
&&&
1\le j < \infty.
\end{align*}
We also define
\begin{align*}
  \nu_j := \lfloor \xi_j/(2\pi)\rfloor,
  \qquad\qquad\qquad
  \zeta_j := \xi_j -2\pi \nu_j.
\end{align*}
Hence, $\nu_j$ and $\zeta_j$ are independent and distributed as
\begin{align*}
\nu_j
&
\sim
{\tt GEOM}(e^{-2\pi}), 
&&
\probab{\nu_j=n} 
=
e^{-2\pi n}(1-e^{-2\pi})
\one(n\geq0)
&&
1\leq j< \infty,
\\
\zeta_j
&
\sim
{\tt TRUNCEXP}(1, 2\pi),
&&
\frac{d}{dx} \probab{\zeta_j < x} 
= 
\frac{e^{-x}}{1-e^{-2\pi}}\one(0 \le x \le 2\pi),
&&
0\leq j< \infty.
\end{align*}
It is convenient to start at time $0$ with a fresh scattering. The construction will be done piece-wise, in the intervals defined by the renewal times $((\tau_{n,k})_{0\le k \le \nu_{n+1}})_{0\le n < \infty}$, as follows. 
\begin{align}\label{renewal}
& 
\tau_{n}
:=
\sum_{j=0}^{n-1}\xi_j, 
&&
\tau_{n,k}
:=
\begin{cases}
\tau_n+ 2\pi k & \text{ for } 0\leq k\leq \nu_{n},
\\
\tau_{n+1} & \text{ for }  k= \nu_{n} +1,
\\
\text{not defined}  & \text{ for }  k> \nu_{n} +1,
\end{cases}
&&
0\leq n<\infty. 
\end{align}
The times $(\tau_n)_{n\geq1}$ will be the times of collisions with \emph{fresh} scatterers not yet seen in the past. While the times $(\tau_{n,k})_{1\leq k\leq \nu_n}$ will be the times of \emph{recollisions} with the scatterer first seen at $\tau_n$.  The process $t\mapsto \varphi(t)\in\R/(2\pi)$ will be piece-wise linear, continuous from the left, with jump discontinuities at the scattering times $((\tau_{n,k})_{0\leq k\leq \nu_n})_{0\le n < \infty}$. For $0\le n < \infty$, let 
\begin{align} \label{recursion}
  \begin{aligned}
    &\varphi_{n+1} : = \varphi_n +(\nu_{n+1}+1)\alpha_{n+1} + \zeta_{n+1}, \\
    &\varphi_{n,k}^+:= \varphi_n +(k+1)\alpha_{n+1}
  \end{aligned}
  \qquad
  0 \le k \le \nu_{n+1}.
\end{align}
Then, for $t\in (\tau_{n,k},\tau_{n,k+1}]$, with $0\le n < \infty$, $0\le k \le \nu_{n+1}$,
\begin{align*}
  \varphi(t):=\varphi_{n,k}^+ +(t-\tau_{n,k})
\end{align*}
Given this, the velocity and position processes are defined by \eqref{U and Y}.


Since we are interested in the \emph{diffusive scaling limit} $N^{-1/2} Y(Nt)$, noting that 
\begin{align*}
\max_{\tau_{n-1}\leq t \leq \tau_{n+1}} \abs{Y(t)-Y(\tau_n)}\leq 2, 
\end{align*}
from now on we will follow only the discrete time processes
\begin{align*}
&
\varphi_n:=\varphi(\tau_n),
&&
U_n:=U(\tau_n)= 
e^{i\varphi_n}, 
&&
Y_n:= Y(\tau_n), 
&&
y_n:=Y_{n}-Y_{n-1}.
\end{align*}
We show that the successive steps $y_n$ are a rather explicit function of a rather explicit and well-behaved Markov chain.

Using \eqref{U and Y} and \eqref{recursion} we readily obtain
\begin{align}
\label{yn}
\begin{aligned}    
y_{n}
&
=
2\sin(\frac{\zeta_{n}}{2}) 
U_n 
e^{-i\frac{\zeta_{n}}{2}},
=
2\sin(\frac{\zeta_{n}}{2}) 
U_{n-1} 
e^{i((\nu_{n}+1)\alpha_{n}+\frac{\zeta_{n}}{2})}.
\end{aligned}
\end{align}
It is convenient to write $y_n = r_n e^{i\theta_n}$ with
\begin{align*}
r_n:= \abs{y_n}
\in[0,2], 
\qquad\qquad
\theta_n:=\arg(y_n)
\in[0,2\pi]. 
\end{align*}
From \eqref{yn}, we readily get
\begin{align}
\label{mc}
\theta_{1}
=
\varphi_0 +(\nu_1 +1)\alpha_1 +\frac{\zeta_{1}}{2},
\qquad
\theta_{n+1}
=
\underbrace{\theta_{n} + \frac{\zeta_n}{2}}_{\varphi_n}
+
(\nu_{n+1} +1)\alpha_{n+1} + \frac{\zeta_{n+1}}{2},
\end{align}
where the sums on the right hand side are meant $\mod{2\pi}$. This recursion shows that $(\zeta_{n+1},\theta_{n+1})$ is determined by $(\zeta_{n},\theta_{n})$ and the \emph{independently sampled} $(\xi_{n+1}, \alpha_{n+1})$. Thus $n \mapsto (\zeta_n,\theta_n)$ is a rather explicit Markov chain on the state space $[0,2\pi)\times [0,2\pi)$. 
\begin{proposition}
\label{prop:Doeblin}
\begin{enumerate}[label = (\roman*)]
   \item The distribution
     \begin{align}\label{statimeasure}
       (\zeta_{n},\theta_n) \sim (\wh{\zeta},\wh{\theta}),
     \end{align}
     where $\wh{\zeta}$ and $\wh{\theta}$ are independent random variables distributed as $\wh{\zeta} \sim {\tt TRUNCEXP}(1,\-2\pi)$ and $\wh{\theta} \sim {\tt UNI}[0,2\pi]$, is stationary for the Markov chain $(\zeta_n,\theta_n)$.
   \item Moreover, D\"oblin's condition holds: There exists a $\delta>0$ such that for any measurable subset $A\subset [0,2\pi) \times[0,2\pi)$
\begin{align}
\label{doeblin}
\condprobab
{(\zeta_{n+1}, \theta_{n+1}) \in A}
{(\zeta_n, \theta_n)} 
\geq 
\delta 
\probab{(\wh \zeta, \wh \theta)\in A}, 
\end{align}
uniformly in the condition $(\zeta_n , \theta_n) \in [0,2\pi)\times [0,2\pi)$. In particular, the distribution \eqref{statimeasure} is the unique stationary measure of the Markov chain $(\zeta_n,\theta_n)$.
\end{enumerate}
\end{proposition}
\begin{proof}
  \emph{(i)} follows from the explicit expressions in \eqref{mc}. Indeed, assuming $(\zeta_n,\theta_n) \sim (\wh{\zeta},\wh{\theta})$, it follows that $\theta \sim {\tt UNI}(0,2\pi)$ is independent of the pair $(\zeta_{n+1}, \frac{\zeta_n}{2} + (\nu_{n+1} +1) \alpha_{n+1} + \frac{\zeta_{n+1}}{2})$. Thus, $\theta_{n+1} \sim {\tt UNI}(0,2\pi)$ and is independent of $\zeta_{n+1}$ and hence $(\zeta_{n+1}, \theta_{n+1})\sim (\wh{\zeta}, \wh{\theta})$, as well. For \emph{(ii)}, the random variable $(\nu_{n+1}+1)\alpha_{n+1} \mod{2\pi}$ is independent of the pair $(\zeta_{n+1}, \theta_n + \frac{\zeta_n}{2} + \frac{\zeta_{n+1}}{2})$ and its distribution density (on the interval $[0,2\pi]$) is bounded from below (note the importance of the fact we are working modulo $2\pi$). Hence \eqref{doeblin}.

\end{proof}

\subsection{The Markovized version of the MLP}
\label{ss: The markovized version of the MLP}

In this section we construct a Markovized version of the MLP. The construction differs from the construction of the limit process, in section \ref{ss: The limit process}, as the scatterers now have radii $\varepsilon>0$. This causes some formal complications. However, the essence is very similar. 

In plain words the Markovized process flies for a truncated exponential time, then collides with a 'fresh scatterer'. After a collision with a fresh scatterer, the Markovized process forgets about all previous scatterers except the scatterer of the previous collision. The process then behaves exactly like a physical Lorentz process with only one scatterer in the plane. Then after an exponential time, regardless of the physical limitations, the Markovized process collides with another fresh scatterer and forgets about all previous scatterers. We illustrate a sample trajectory for the Markovized MLP in Figure \ref{fig:markovized}.


\begin{figure}[t!]
  \begin{center}    
    \includegraphics[width=0.8\textwidth]{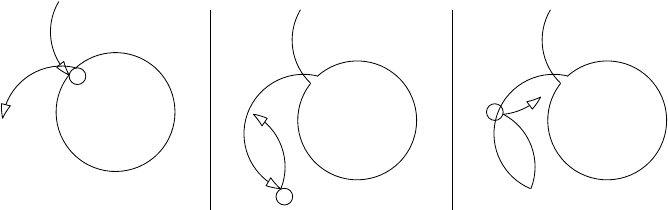}
  \end{center}
  \caption{%
    {\tt We illustrate a sample trajectory of the Markovized version of the MLP at three successive times. Note that the third panel would not be possible for the physical MLP, since the scatterer is placed in a region of space which was already explored (a shadowed collision). However, the Markovized version ignores this physical complication and collides anyways.
    }   
  }%
  
  \label{fig:markovized}

\end{figure}

The probabilistic ingredients used in this construction are the same as in the previous one:  $\varphi_0 \sim {\tt UNI}[0, 2\pi]$, $(\xi_j)_{j \ge 1} \sim {\tt EXP}(1)$ and $(\alpha_j)_{j \ge 1} \sim 2 \arccos({\tt UNI}[-1,1])$, all fully independent. We also define the variables $\beta_j\in[0,2\varepsilon]$, $\gamma_j\in[0,2\pi]$ by the formulas
\begin{align}
\label{beta-gamma-eps}
\begin{aligned}
&
\beta_j
=
\beta(\alpha_j,\vareps)
:=
2\arctan\frac{\vareps \sin(\alpha_j/2)}{1+\varepsilon\cos (\alpha_j/2)}, 
\\
&
\gamma_j
=
\gamma(\alpha_j,\vareps)
:=
2\arctan\frac{\sin(\alpha_j/2)}{\cos (\alpha_j/2)+\varepsilon}.
\end{aligned}
\end{align}
and write 
\begin{align}
\label{exp=geom+trexp-eps}
\begin{aligned}
&
\nu_j
=
\nu(\xi_j, \alpha_j, \vareps)
:= 
\lfloor\xi_j/(2\pi-\beta_j)\rfloor, 
\\
&
\zeta_j
=
\zeta(\xi_j, \alpha_j, \vareps)
:= 
\xi_j-\nu_j(2\pi-\beta_j). 
\end{aligned}
\end{align}
Note that, \emph{conditionally on $\alpha_j$} (and thus, on the value of $\beta_j$),  the random variables $\nu_j$ and $\zeta_j$ are independent and distributed as
\begin{align}
\label{geom+trexp-distrib-eps}
\begin{aligned}
\probab{\nu_j=n} 
&
=
e^{-(2\pi-\beta_j)n}(1-e^{-(2\pi-\beta_j)})\one(n\geq0)
\sim 
{\tt GEOM}(e^{-(2\pi-\beta_j)}), 
\\
\frac{d}{dx} \probab{\zeta_j < x} 
&
=
\frac{e^{-x}}{1-e^{-(2\pi-\beta_j)}}\one(0 \le x \le 2\pi-\beta_j)
\sim
{\tt TRUNCEXP}(1,2\pi-\beta_j). 
\end{aligned}
\end{align}
We will also use 
\begin{align}
\label{epsi-eps}
\epsilon_j:= 
\epsilon(\xi_j, \alpha_j,\varepsilon)
:=
1-2\lfloor \zeta_j/\pi \rfloor.
\end{align}
The  renewal times $\tau_n$ are the same as in \eqref{renewal}.  

We will use the same complex notation as in the previous section and construct inductively the variables
\begin{align}
\label{YnUn-eps}
&
U^\vep(\tau_n^-)=:U_n^\vep=:e^{i\varphi_n}, 
&&
Y^\vep(\tau_n)=:Y^\vep_n, 
&&
Y^\vep_{n}-Y^\vep_{n-1}=:y_n=:\wh y_n+\wt y_n.
\end{align}
using the collection of fully independent random variables $\left(\varphi_0, (\xi_n)_{n\ge 1}, (\alpha_n)_{n\ge 1}\right)$. 
Let $Y^\vep_0:=0$, $U^\vep_0:=e^{i\varphi_0}$ and proceed inductively as shown in \eqref{limit comps-eps} below. 
To distinguish between $(\phi \text{ mod } 2\pi)/2$ and $(\phi/2) \text{ mod } 2\pi$ we will use the notation
\begin{align}
\frac{\{\phi\}}{2}
:=
\frac{\phi-2\pi\lfloor\phi/2\pi\rfloor}{2}
\end{align} 
 
\begin{align}
\label{limit comps-eps}
\begin{aligned}
\varphi_{n+1} 
&
:= 
\varphi_n +(\nu_{n+1}+1)\alpha_{n+1} -\nu_{n+1}\beta_{n+1} + \zeta_{n+1}
\\
U^\vep_{n+1}
&
:=
e^{i\varphi_{n+1}}
=
e^{i(\varphi_n + (\nu_{n+1}+1)\alpha_{n+1} - \nu_{n+1}\beta_{n+1} + \zeta_{n+1}}
\\ 
\wh y_{n+1}
&
:=
2 \sin(\frac{\{\nu_{n+1}\gamma_{n+1}\}}{2}) 
e^{i (\varphi_n + \frac{\alpha_{n+1}}{2} + \frac{\{\nu_{n+1}\gamma_{n+1}\}}{2} )}
\\ 
\wt y_{n+1}
&
:=
2\sin(\frac{\zeta_{n+1}}{2}) 
e^{i(\varphi_n + (\nu_{n+1}+1)\alpha_{n+1} - \nu_{n+1}\beta_{n+1} + \frac{\zeta_{n+1}}{2})}
\\
Y^\vep_{n+1}
&
:=
Y_n^\vep+\vep\wh y_{n+1}+\wt y_{n+1}
\end{aligned}
\end{align} 
One can easily check that these formulas cover exactly what was said in plain words in an earlier paragraph. The continuous time evolution $t\mapsto (U^\vep(t), Y^\vep(t))$ within the time intervals $[\tau_n, \tau_{n+1})$ also has explicit expression. However, as we don't need it for the rest of the arguments we will not give it here. 

It is convenient to write
\begin{align}
\label{convenient-eps}
\begin{aligned}
& 
\wh r_n:= \abs{\wh y_n}, 
&& 
\wh \theta_n:=\arg(\wh y_n), 
&&
\wh y_n=\wh r_n e^{i\wh \theta_n},
\\ 
& 
\wt r_n:= \abs{\wt y_n}, 
&& 
\wt \theta_n:=\arg(\wt y_n), 
&&
\wt y_n=\wt r_n e^{i\wt \theta_n},
\end{aligned}
\end{align}
Then the recursion \eqref{limit comps-eps} is rewritten as
\begin{align}
\label{recursion-eps}
\begin{aligned}
&
\epsilon_{1}
=
\epsilon(\frac{\zeta_{1}}{2}), 
&&
\wt{r}_{1}
=
2\sin(\frac{\zeta_{1}}{2}), 
&&
\wt{\theta}_{1}
=
\varphi_0+(\nu_{1}+1)\alpha_{1}+\frac{\zeta_{1}}{2}
\\
&
\epsilon_{n+1}
=
\epsilon(\frac{\zeta_{n+1}}{2}), 
&&
\wt{r}_{n+1}
=
2 \sin(\frac{\zeta_{n+1}}{2}), 
&&
\wt{\theta}_{n+1}
=
\wt{\theta}_{n}+\eps_{n}\arcsin(\wt{r}_n/2)+(\nu_{n+1}+1)\alpha_{n+1}+\frac{\zeta_{n+1}}{2},
\end{aligned}
\end{align}
 and, recalling that $\beta_n+\gamma_n = \alpha_n$
\begin{gather}
\label{recursion-eps wh}
\begin{gathered}
\wh{r}_{1}
=
2\sin(\{\nu_{1} \gamma_1\}/2), 
\qquad
\wh{\theta}_{1}
=
\varphi_0+\alpha_{1}/2+\{\nu_1\gamma_1\}/2
\\
\wh{r}_{n+1}
=
2\sin(\{\nu_{n+1} \gamma_{n+1}\}/2), 
\\
\wh{\theta}_{n+1}
=
\begin{cases}
  0 & \mbox{ if } \nu_{n+1}=0\\
  \wt{\theta}_{n} + \epsilon_n \arcsin(\frac{\wt{r}_n}{2}) + \frac{\alpha_{n+1}}{2} + \arcsin(\frac{\wh{r}_{n+1}}{2}) & \mbox{otherwise}
\end{cases}
\end{gathered}
\end{gather}
Thus, since $\zeta_n$ can be written in terms of $\epsilon_n$ and $\wt{r}_n$, we conclude that
\begin{align*}
  (\epsilon_{n+1}, \wt{r}_{n+1},\wt{\theta}_{n+1},\wh{r}_{n+1},\wh{\theta}_{n+1})
\end{align*}
is determined by $(\epsilon_{n}, \wt{r}_{n},\wt{\theta}_{n},\wh{r}_{n},\wh{\theta}_{n})$ and the independently sampled $(\alpha_{n+1},\xi_{n+1})$. Thus, we have a Markov chain on the state space $\{-1,1\}\times [0,2]\times[0,2\pi]\times[0,2]\times [0,2\pi]$. We denote this Markov chain $\psi_n^\vep$.

Note that, if we condition on $\nu_{n+1}=0$, then $\wh{r}_{n+1}=0$ and $\wh{\theta}_{n+1}=0$. In which case the Markov process is identical to the Markov process in Section \ref{ss: The limit process}, hence D\"oblin's condition applies also in this case with respect to the random variable
\begin{align*}
  \Psi(\theta, \zeta) = (1-2\lfloor \zeta/\pi\rfloor, 2 \sin (\zeta/2),\theta, 0,0)
\end{align*}
where $\theta \sim {\tt UNI}[0,2\pi]$, and $\zeta \sim {\tt TRUNCEXP}(1,2\pi)$. By Proposition \ref{prop:Doeblin} the following proposition is immediate.

\begin{proposition}\label{prop:Doeblin eps}
   The Markov chain $n \mapsto \psi_n^\vep$ satisfies the D\"{o}blin condition uniformly in $\varepsilon$ with respect to $\Psi$: There exists a $\delta >0$ such that for any $\varepsilon>0$ sufficiently small and any measurable subset $A \subset  \{-1,1\}\times [0,2]\times[0,2\pi]\times[0,2]\times [0,2\pi]$
  \begin{align}\label{Doeb vep}
    \condprobab{\psi^\vep_{n+1} \in A}{\psi^\vep_n} \ge \delta \probab{\Psi \in A },
  \end{align}
  uniformly in the condition. 
\end{proposition}

\subsubsection{Magnetic Lorentz process}
  \label{ss:Magnetic Lorentz Process}

  Now we define the physical, MLP process $\{t \mapsto X^\vep(t)\}$ coupled to $\{t \mapsto Y^\vep(t)\}$. The idea is the following: we will have $X^\vep(0) = Y^\vep(0)=0$ and $X^\vep(t) = Y^\vep(t)$ for all $t>0$ until the first time the $Y^\vep$-process does something which is forbidden in the dynamics of the physical process. At which point, we simply stop the $X$-process. In the subsequent sections we show that, with high probability, this stopping time will occur at time $\vep^{-1} \abs{\log \vep}^{-1}$. Therefore until time $T(\vep) = o(\vep^{-1}\abs{\log \vep}^{-1})$ both processes are (w.h.p) well-defined. 

  Define the following indicator functions:
  \begin{align}\label{mismatches}
    \begin{aligned}
      & \wh{\eta}_j := \one\left( \left\{\exists t < \tau_{j-1} : \abs{Y_{j}^{\vep} - Y^\vep(t)}< \vep \right\}\right) \\
      & \wt{\eta}_j := \one\left( \left\{\exists i < j-1 \quad \& \quad t \in [\tau_{j-1},\tau_j) : \abs{Y_{i}^\vep - Y^\vep(t)} < \vep \right\}\right)        \\
      & \eta_j      := \max\{\wt{\eta}_j,\wh{\eta}_j\}.
    \end{aligned}
  \end{align}
  That is, $\wh{\eta}_j$ is the indicator of the event that the $Y^\vep$-process performs a jump and the scatterer corresponding to that jump blocks a previously explored region. We call this a \emph{shadowed scattering}. $\wt{\eta}_j$ is the indicator of the event that, during the interval $[\tau_{j-1},\tau_j)$, the $Y^\vep$-process \emph{recollides} with a previously placed scatterer \emph{other than the 'current' scatterer} (i.e a recollision with a scatterer other than the scatterer at position $Y_{j-1}^\vep$). Note that the $Y^\vep$-process ignores this \emph{recollision with a past scatterer}. $\eta_j$ is the indicator that either of these \emph{mismatches} occurs for the $j^{th}$ path segment. These mismatches are illustrated in Fig. \ref{fig:mismatch}.


\begin{figure}[t!]
  \begin{center}    
    \includegraphics[width=0.7\textwidth]{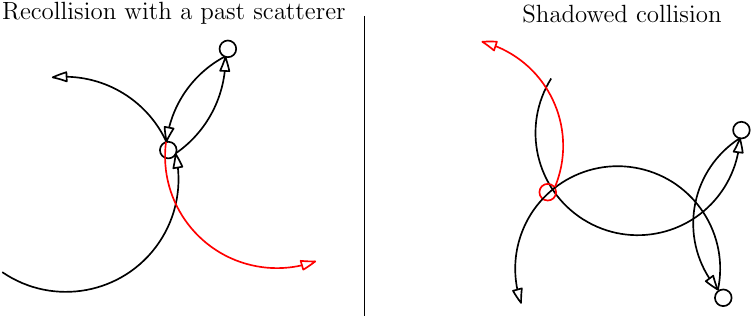}
  \end{center}
  \caption{%
    {\tt On the left we give an example of a recollision with a past scatterer. The physical MLP (black) must respect this collision. Whereas the Markovized MLP (red) has forgotten about the past scatterer and thus ignores the recollision. On the right we give an example of a shadowed collision. Here the Markovized MLP (red) changes direction in the third collision. However this scatterer is in previously explored space, thus the physical MLP must ignore the collision and continues in its trajectory. We call this a shadowed collision. 
    }   
  }%
  
  \label{fig:mismatch}

\end{figure}

    Now we define the following stopping time 
    \begin{align}\label{rho}
      \rho := \min \{ n :\eta_n = 1 \}.
    \end{align}
    For $t \in [0,\tau_{\rho-1})$ we set $X^\vep(t) = Y^\vep(t)$. Thus, the mechanical $X^\vep$-process is defined up to the first collision \emph{before a mismatch}. It would be possible to include a recoupling procedure (as we did for the classical Lorentz gas in $3$ dimensions \cite{lutsko-toth_19}) allowing the $X^\vep$ process to be defined after this stopping time. However, our techniques do not allow us to extend the invariance principle to these times (this is a result of the recurrence of the random walk in $2$ dimensions), hence there is no need to over-complicate the exposition by defining $X^\vep$ beyond $\tau_{\rho-1}$.

\subsubsection{Main technical result}
\label{ss:Main Technical Result}

The main technical theorem which implies Theorem \ref{thm:main-theorem} is


\begin{theorem}\label{thm:XY}
  Let $T=T(\vep)$ be such that $\lim_{\vep\to 0}T(\vep) = \infty$ and $\lim_{\vep \to 0 } T(\vep) \vep \abs{\log \vep} =0$. Then
  \begin{align}\label{X not Y}
    \lim_{\vep \to 0} \probab{ \tau_{\rho}\le T} = 0.
  \end{align}

\end{theorem}

The proof of Theorem \ref{thm:XY} follows two steps: First, we use D\"oblin's inequality to decompose the process $\{t\mapsto Y^\vep \}$ into i.i.d legs. These legs will form a  stationary random walk. This decomposition then allows for control of the associated Green's function, which gives a bound for the probability of a mismatch between two different legs. The second step is to then use geometric arguments to show that within each separate leg there are no mismatches. From there, Theorem \ref{thm:XY} follows immediately. 

\section{Proof}
\label{s: Proof}

The proof of Theorem \ref{thm:XY} depends only on $\{t \mapsto Y^\vep\}$. To ease the notation, we henceforth drop the $\vep$ from the label.

\subsection{Decomposition into legs}
\label{ss: Decomposition into legs}

There are two issues which prevent us from directly applying some standard probabilistic techniques for Green's functions. The first is the 'self-recollisions' (i.e when $\nu > 0$). These complicate certain probabilistic arguments we make later. The second is the Markov chain itself, which we want to separate into independent pieces. To address these two issues we decompose the random data into 'packs'. On each of these packs, we construct a 'leg' of the Markovized process. The legs are then i.i.d.

\subsubsection{D\"oblin condition}
\label{ss:Doeblin Condition}

Consider the Markov chain $n \mapsto \psi_n^\vep =: \psi_n$, we will use a standard trick to decompose this Markov chain into a random walk with independent steps. Consider the process:

Let $P(\psi_n , \psi_{n+1})$ denote the transition probability of this Markov chain. By Proposition \ref{prop:Doeblin eps} we have that, there exists a $\delta>0$ such that
\begin{align}\label{Doeblin}
  P(\psi_n, \psi_{n+1}) > \delta \pi(\psi_{n+1})
\end{align}
where $\pi$ is the distribution of $\Psi$. Hence, we can express
\begin{align}\label{Doeblin 2}
  P(\psi_n,\psi_{n+1}) = \delta \pi(\psi_{n+1}) +(1-\delta)Q(\psi_n,\psi_{n+1})
\end{align}
where $Q$ is some other probability transition function. 

The advantage of \eqref{Doeblin 2} is that we can reconstruct the Markov chain $(\psi_n)_{n \ge 1}$ as follows: first, define a Bernoulli random variable:
\begin{align*}
  b_n \sim {\tt BERN}(\delta), \qquad \probab{b_n = 1} = \delta, \qquad \probab{b_n =0} = 1-\delta, \qquad (n \ge 2)
\end{align*}
and $b_1=1$. If $b_n =1$, then $\psi_n$ is distributed according to $\pi$, if $b_n=0$ then $\psi_n$ is distributed according to $Q(\psi_{n-1},\psi_n)$.

\subsubsection{Breaking $(\psi_n)_{n\ge 1}$ into packs}
\label{ss:Breaking}

To construct the legs, we break the random data into 'packs'. In words, the first and last path segments of each pack involve no recollisions (i.e $\nu =0$). Moreover, to start and end a pack, we require that $b_n =1$ (that is, we use the distribution $\pi$ rather than $Q$ to select $\psi_n$).

Formally, let $\Gamma_0=1$ and let $\Theta_0 = 0$. Now for $n \ge 1$ set:
\begin{align*}
  &\Gamma_n := \min \left\{ j >  \Gamma_{n-1} : b_{j-1}=b_j =1, \nu_{j-1}=\nu_{j}= 0 \right \},
  &&\qquad \gamma_n := \Gamma_n -\Gamma_{n-1},\\
  &\Theta_{n} := \tau_{\Gamma_n},
  &&\qquad\theta_n := \Theta_n - \Theta_{n-1},
\end{align*}
the condition that the $(j-1)^{st}$ and $j^{th}$ steps involve no self-recollisions is necessary since, in Proposition \ref{prop:Doeblin eps} the random variable $\Psi$ assumes no recollisions. Furthermore for $n \ge 1$ let:
\begin{align*}
  &\xi_{n,j} = \xi_{\Gamma_{n-1}+j-1},
  &&  1\le j \le \gamma_{n},\\
  &y_{n,j} = y_{\Gamma_{n-1}  +j-1},
  &&  1 \le j \le \gamma_{n}.
\end{align*}  

We define a \emph{pack} to be:
\begin{align*}
  &\varpi_n : = (\gamma_n , (\xi_{n,j}, y_{n,j}): 1 \le j \le \gamma_n ) \qquad n \ge 1.
\end{align*}
It is important to note that the individual packs are \emph{independent, identically distributed}. On each pack we define the associated discrete and continuous $Y=Y^\vep$ process: for $n \ge 1$:
\begin{align*}
  &Y_{n,0} =Y(\Theta_{n-1}),   && Y_{n,j} = Y_{n,0} + \sum_{i = 1}^j y_{n,j},      &&j \le \gamma_n, \\
  &Y_n(0)=Y(\Theta_{n-1}),    && Y_n(t) := Y(\Theta_{n-1}+t),              && 0\le t \le \theta_n. 
\end{align*}
Finally we define the discrete \emph{end-point process}:
\begin{align*}
  \Xi_0 = 0,  \qquad   \qquad \Xi_n := \Xi_{n-1} + \sum_{i=1}^{\gamma_n}y_{n,i}, \qquad  \qquad n \ge 1.
\end{align*}
That is the process $\{n \mapsto \Xi_n\}$ denotes the discrete process made of the end-point of the successive legs.

\subsubsection{Green's function estimates}

Define the following Green's functions (occupation measures): given a set $A\subset \R^2$,
\begin{align*}
  & G(A) := \expect{ \abs{ \{ 1 \le k \le N( T ) : Y_k \in A \}}}\\
  & H(A) := \expect{ \abs{ \{0 < t \le T : Y(t) \in A \}}}\\
  & g(A) := \expect{ \abs{ \{ 1 \le k \le \gamma_1 : Y_k \in A \}}}\\
  & h(A) := \expect{ \abs{ \{0 < t \le \theta_1 : Y(t) \in A \}}}\\
  & R(A) := \expect{ \abs{ \{ 1 \le k \le n(T) : \Xi_k \in A \}}},\\
\end{align*}
where $N(T) = \max\{ n\ge 0 : \tau_n < T\}$ (the total number of collisions in time $T$) and $n(T) := \max\{ n \ge 0 :  \Theta_{\Gamma_n}<T\}$ (the total number of legs).

These occupation measures then satisfy the following bounds


\begin{lemma}\label{lem:greensbnds}
  For any measurable set $A \subset \R^2$ the following bounds hold:
  \begin{gather}\label{green_bounds}
    \begin{gathered}
      g(A), h(A) \le L(A), \qquad R(A), G(A), H(A) \le K(A) + L(A)
    \end{gathered}
  \end{gather}
  where

  \begin{gather*}
    K(dx) := C\min \{\vep^{-1},\abs{x}^{-1}\}dx + Cdx     , \qquad L(dx) := C\abs{x}^{-1}e^{-c\abs{x}}dx
  \end{gather*}
  for some constants $0<c,C<\infty$.
\end{lemma}

\begin{proof}

  First, we consider
  \begin{align}
    g_1(A):= \probab{y_1 \in A}.
  \end{align}
  Since $\nu_1 = 0$ we know that
  \begin{align*}
    y_{1}= \wt{y}_1 = U_0 e^{i\alpha_1} 2 \sin \frac{\xi_1}{2} e^{i \xi_1/2},
  \end{align*}
   in particular, since the density of distribution of $\alpha_1$ and $\xi_1$ are bounded, there exists a constant $C<\infty$ such that
  \begin{align}\label{g1 bound}
    g_1(A) \le C\probab{ v \in A } = : C\wt{g}_1(A)
  \end{align}
  where $v$ is a uniformly distributed vector on the radius $2$ disk. 

  To achieve the desired bound we introduce an auxiliary discrete process, $(\wt{Y}_n)_{1\le n \le \gamma_1}$. Let $\wt{Y}_0=0$ and let $\wt{Y}_1 = v$. Then, let  $\wt{Y}_2-\wt{Y}_1$ be distributed like $y_1$ (independent of $\wt{Y}_1$), and all subsequent steps follow the construction of the markovized Lorentz process (see Subsection \ref{ss: The markovized version of the MLP}). In words, $n \mapsto \wt{Y}_n$ takes one step which is simply a uniformly distributed vector on the radius $2$ disk. Then, from there we start a discrete markovized Lorentz process. With that, the following upper bound follows from \eqref{g1 bound} and the bounded distributions:
  \begin{align}
    \begin{aligned}\label{convolution}
    g(A) &= \expect{ \condexpect{\abs{\{1 \le k \le \gamma_1 : Y_k - Y_1 \in A-x \}}}{Y_1=x}} \\
    &\le C \expect{  \condexpect{\abs{\{1 \le k \le \gamma_1 : \wt{Y}_k - \wt{Y}_1 \in A-x \}}}{\wt{Y}_1 =x }} \\
    &\le C \int_{\R^2} \wt{g}_2(A-x) \wt{g}_1(dx).
    \end{aligned}
  \end{align}
  where $\wt{g}_2(A) : = \expect{\abs{\{1 \le k \le \gamma_1 : \wt{Y}_k-\wt{Y}_1  \in A \}}}$, that is, we use the independence of the first two steps of $\wt{Y}_k$ to decompose the above expectation into a convolution.

 Because of the geometric distribution of $b_n$ we conclude that $\gamma_1$ is exponentially tight. That is,
  \begin{align*}
    \wt{g}_2(\{x: \abs{x} > s \}) \le Ce^{-cs}, \qquad \qquad
    \wt{g}_2(\R^2) = \expect{\gamma_1} < \infty.
  \end{align*}
  Inserting these into \eqref{convolution}, and using the fact that $\wt{g}_1(dx) \le \abs{x}^{-1}dx$, the bounds in \eqref{green_bounds} on $g(A)$ follow. The bound on $h(A)$ follows more or less immediately from the bound on $g(A)$ and the fact that the $\xi_n$ are exponentially distributed.

  Turning now to $R(dx)$, recall that $\Xi$ is a random walk, whose steps are i.i.d and exponentially bounded. Standard bounds on such walks are well-known (see e.g {\cite[Lemma 6]{lutsko-toth_19}}) yielding the bound $R(dx) \le \abs{x}^{-(d-2)} dx$ as $x \to \infty$. To achieve a bound when $x\to 0$ we use a similar strategy as for $g(dx)$. Define a discrete process $\{\wt{\Xi}\}_{n \in \N}$ similar to $\{\Xi_n \}$ as follows: $\wt \Xi_1 \sim \Xi_1$, then $\wt\Xi_2 - \wt\Xi_1 \sim \Xi_1$ independent of $\wt\Xi_1$. All the rest of the $\wt\Xi_n$ for $n \ge 3$ are distributed according to the rules defining $\Xi_n$. Now, once again we can use the fact that the density of distribution of the first step of $\Xi_2 -\Xi_1$ (i.e $y_{2,1}$) is bounded to conclude: 
  \begin{align} \label{R conv}
    \begin{aligned}
      R(A) &\le C\probab{\wt\Xi_1 \in A} + C\int_{\R^2} \wt{R}(A-x) \probab{\wt\Xi_1 \in dx},\\
           &= C\probab{\Xi_1 \in A} + C\int_{\R^2} R(A-x) \probab{\Xi_1 \in dx},
    \end{aligned}
  \end{align}
  where $\wt{R}(A) : = \expect{ \abs{ \{ 2 \le k < n(T) : \wt{\Xi}_k - \wt{\Xi}_1 \in A \}}}$. Now note that by the definition of $n(T)$ we have $R(\R^2) = \expect{n(T)} = C \vep^{-1}$. With that the bound, \eqref{green_bounds} for $R$ follows from $\probab{\Xi_1 \in A} < g(A)$ and \eqref{R conv}

  The bounds for $G$ and $H$ follow a similar strategy. We use the bounded density of distribution of the first step to write
  \begin{align}
    \begin{aligned}
    G(A) &\le C g(A) + C\int_{\R^2} g(A-x) R(dx),\\
    H(A) &\le C h(A) + C\int_{\R^2}h(A-x) R(dx),
    \end{aligned}
  \end{align}
  from which \eqref{green_bounds} follow.
\end{proof}

\subsection{Inter-leg mismatches}
\label{ss: Inter-leg mismatches}

In this section we show that there are no mismatches between path segments in \emph{different} legs. Let $t \mapsto Y(t)$ be a Markovian flight process starting at the origin with initial velocity $U_0$. Now we construct two auxiliary processes. Let $\wt{\alpha} \sim 2 \arccos({\tt UNI}[-1,1])$. Now let  $t \mapsto Y^\ast(t)$ denote a Markovian flight process moving in the clockwise (rather than anti-clockwise) direction, starting at the origin, with initial velocity $U_0 e^{\wt{\alpha}}$. In words, the process $t\mapsto Y^{\ast}(t)$ denotes the past trajectory of $t \mapsto Y(t)$. Finally we let $t \mapsto Y^{\ast\ast}$ be a Markovian flight process moving in the clockwise direction which is fully independent of $Y$ and $Y^{\ast}$. We let $n \mapsto Y^\ast_n$ and $n \mapsto Y^{\ast\ast}_n$ denote the discrete processes associated to $Y^\ast$ and $Y^{\ast\ast}$ respectively. Define the following events:
\begin{align*}
  & \wh{W}_j := \left\{ \min\{ \abs{Y(t) - Y_j} : \quad 0 < t < \Theta_{j-1}, \quad \Gamma_{j-1} < k \le \Gamma_j \} < \vep        \right\} \\
  & \wt{W}_j :=  \left\{ \min\{ \abs{Y_k - Y(t)} : \quad 0 < k < \Gamma_{j-1}, \quad \Theta_{j-1} < t \le \Theta_j \} < \vep        \right\}   \\
  & \wh{W}_T^{\ast}  := \left\{\min\{\abs{Y^{\ast}(t) - Y_k}: \quad 0 < t < T, \quad 0 <k \le \gamma       \}<\vep \right\}\\
  & \wt{W}_T^{\ast} :=   \left\{\min\{\abs{Y^{\ast}_k - Y(t)}: \quad 0 < k < \lfloor T\rfloor, \quad 0 <t \le \theta       \}<\vep \right\}.    \\    
  & \wh{W}_T^{\ast\ast}  := \left\{\min\{\abs{Y^{\ast\ast}(t) - Y}: \quad 0 < t < T, \quad 0 <k \le \gamma       \}<\vep \right\}\\
  & \wt{W}_T^{\ast\ast} :=   \left\{\min\{\abs{Y^{\ast\ast}_k - Y(t)}: \quad 0 < k < \lfloor T\rfloor, \quad 0 <t \le \theta       \}<\vep \right\}.   
\end{align*}
Note that there exists a constant $C< \infty$ independent of $\vep$ such that:
\begin{align}\label{W bounds}
  \begin{aligned}
    &\probab{\wh{W_j}} \le \probab{\wh{W}^{\ast}_T}\le C\probab{\wh{W}^{\ast\ast}_T} \le 2C\vep^{-1} \sum_{z \in \Z^2} G(B_{z\vep,2\vep})h(B_{z\vep,3\vep}) \\
    &\probab{\wt{W_j}} \le \probab{\wt{W}^\ast_T} \le C\probab{\wt{W}^{\ast\ast}_T}  \le 2C\vep^{-1} \sum_{z \in \Z^2} H(B_{z\vep,2\vep})g(B_{z\vep,3\vep}). \\
  \end{aligned}
\end{align}
As before, the second inequality follows from the bounded distribution of the first step of a Markovian flight process.

Now we insert the bounds from Lemma \ref{green_bounds} into \eqref{W bounds} and evaluate the resulting sums using Riemann sums (similar to {\cite[Section 3.4]{lutsko-toth_19}}), i.e
\begin{align*}
  \probab{\wh{W_j}} &\le 2C\vep^{-1} \sum_{z \in \Z^2} G(B_{z\vep,2\vep})h(B_{z\vep,3\vep}) \\
      &\le 2C\vep^{-1} \sum_{z \in \Z^2} (K(B_{z\vep,2\vep}) + L(B_{z\vep,2\vep})) L(B_{z\vep,3\vep})  \\
      &\le C\vep^{3} \sum_{0 \neq z \in \Z^2} (\abs{\vep z}^{-1} +\abs{\vep z}^{-1}e^{-c\abs{\vep z}} )\abs{\vep z}^{-1}e^{-c\abs{\vep z}} + C\vep  \\  
      &\le C\vep^{3} \sum_{0 \neq z\in \Z^2} \abs{\vep z}^{-2} \cdot e^{-c\abs{z}} + C\vep\\
      &\le C\vep \int_{\R^2 \setminus B_{0,\vep}} e^{-c\abs{z}}\abs{z}^{-2}dz + C\vep\\
      &\le C\vep \int_\vep^\infty e^{-c u }u^{-1} du + C\vep
      \le C\vep \abs{\log \vep}
\end{align*}
for constants $C < \infty$ changing from line to line. Thus (with the equivalent bound for $\wt{W}_j$ following the same lines):

\begin{proposition}\label{prop:inter}
  There exists a constant $C<\infty$ such that for all $j\ge 1$
  \begin{align}\label{inter}
    \probab{\wt{W}_j} \le C\vep \abs{\log \vep}, \qquad
    \probab{\wh{W}_j} \le C\vep \abs{\log \vep}.    
  \end{align}

\end{proposition}

\subsection{Intra-leg mismatches}
\label{ss: Intra-leg mismatches}

Above, we have shown that (on our time scales) the probability of a mismatch occurring between path segments in two different legs goes to $0$ as $\vep \to 0$. It remains to show that path segments within a single leg do not interfere with one another. For this we require the following simple geometric lemma.

\subsubsection{A geometric lemma}
\label{ss:A Geometric Lemma}

In this section, we estimate the probability that the Markovian flight process, during one path segment (i.e between two successive fresh collisions) hits a pre-placed spherical region of radius $\vep$. To that end, fix a triple $o \in \R$ (taken for the moment to be the origin), $u \in S_1$, and $\alpha \in [0,2\pi]$. Let $\{t \mapsto \cY(t)\}$ be a Markovian flight process such that $\cY(0) = o$, with initial velocity $ue^{i\alpha}$ and one single scatterer centered at $b \in \vep S^1$. Note that $b$ is uniquely determined by $u$ and $\alpha$ (i.e the scatterer is at the unique position such that collision at $o$ with velocity $u$ results in velocity $ue^{i\alpha}$). Therefore, $\cY$ performs circular motion, colliding infinitely many times with the scatterer at $b$.

\begin{lemma}\label{lem:geometric}
  Let $u \in S_1$ be uniformly distributed on the circle and $\alpha \sim 2 \arccos({\tt UNI}[-1,1])$. There exists a constant $C< \infty$ such that for $\vep>0$ small enough, any time $\xi \in (0,\infty)$, and any vector $v \in \R^2$ 
  \begin{align}\label{geo}
    \probab{ \min_{t<\xi} \abs{\cY(t) - v} <\vep} \le C\frac{\vep}{\abs{v}}\xi.
  \end{align}
  where $\{t \mapsto \cY(t)\}$ is the process described above. 

\end{lemma}

\begin{proof}
  First we consider the trajectory before the first self-collision. It is evident that
  \begin{align}\label{geo 2}
    \probab{ \min_{t< t_0} \abs{\cY(t) - v} <\vep } \le C\frac{\vep}{\abs{v}}.
  \end{align}
  where $t_0$ is the time of the first recollision (i.e $\min (t >0: \nu(t) =1)$). 

  Now we proceed inductively. Let $\zeta(\xi)$ be defined as in \eqref{exp=geom+trexp-eps}
, then:
  \begin{align*}
    \probab{\min_{t\in [\xi - \zeta(\xi), \xi]} \abs{\cY(t) - v } < \vep     } = \probab{\min_{t< \zeta(\xi)} \abs{\wt{\cY}(t) - v } < \vep }
  \end{align*}
  where $\wt{\cY}$ is a process which begins at position $\wh{y}$ (defined as in \eqref{limit comps}) with velocity: $\wt{u} =u e^{i (\nu(\xi)+1)\alpha - i \nu(\xi)\beta}$ where $\beta = \beta(u,\alpha)$ is defined in \eqref{beta-gamma-eps}.

  Note that $\abs{\wh{y}} < \vep$, and $\abs{\beta} < 2\vep$. Therefore, using the uniformity of \eqref{geo 2}, we conclude
  \begin{align*}
    \probab{\min_{t< \zeta(\xi)} \abs{\wt{\cY}(t) - v } < \vep } \le C \frac{\vep}{\abs{v}},
  \end{align*}
  for $C<\infty$ large enough. Now we can use the union bound and the inductive hypothesis to conclude \eqref{geo}.

\end{proof}

\subsubsection{Intra-leg recollisions}
\label{ss:Intra}

The next lemma controls the probability of a mismatch in the first leg recall the indicators in \eqref{mismatches}.

\begin{lemma} \label{lem:mismatch}
  There exists a constant $C<\infty$ such that, for any label $j \in \{1,\dots, \gamma\}$ the following bound holds:
  \begin{align} \label{prob mismatch}
    \expect{\eta_j} \le C\gamma \vep\abs{\log \vep} .
  \end{align}
\end{lemma}

\begin{proof}
  First note that, under time reversal, recollisions become shadowed events, and vice-versa. Therefore, since the distribution of the $Y$-process is invariant under time reversal (up to a mirror symmetry), we may focus only on recollisions. Next, define the following indicators:
  \begin{align}
    \wt{\eta}_{j,k} := \one\left( \{\exists t\in [\tau_{j-1},\tau_j) : \abs{Y_k^\prime - Y(t)} < \vep \}\right)
  \end{align}
  for $j = 3, \dots, \gamma$ and $k \le j -2$. That is, $\wt{\eta}_{j,k}$ is the indicator that during the $j^{th}$ path segment, the point particle recollides with the $k^{th}$ scatterer. 

  For simplicity, we begin by considering \emph{direct} recollisions, when $k = j-2$. The probability of a direct recollision can be controlled rather easily using Lemma \ref{lem:geometric}. Namely, if we condition on $\xi_j, Y_{j-1}, Y_{j-2}$, then Lemma \ref{lem:geometric} implies:
  \begin{align*}
    \condexpect{\wt{\eta}_{j,j-2}}{\xi_j, Y_{j-1}, Y_{j-2}} \le \min\left(\frac{C\vep}{\abs{Y_{j-1} - Y_{j-2}}} (\nu_j+1), 1\right)
  \end{align*}
  Note that Lemma \ref{lem:geometric} was proved for uniformly distributed velocities, however since the collision kernel has bounded density it can be applied here as well, and the only change will be a uniform constant multiple. 
  
  A simple geometric observation is that there exists a constant $C>0 $ such that $\abs{Y_{j-1} - Y_{j-2}} \ge C \min(\zeta_{j-1}, 2\pi-\beta_{j-1} - \zeta_{j-1})$. Thus:
  \begin{align*}
    \condexpect{\wt{\eta}_{j,j-2}}{\xi_j, \xi_{j-1}} \le \min \left(\frac{C\vep}{\min(\zeta_{j-1}, 2\pi-\beta_{j-1} - \zeta_{j-1})} (\nu_{j}+1), 1\right)
  \end{align*}
  Now we can use the tower rule and the distributions of $\zeta_{j-1}$ and $\nu_j$ to bound
  \begin{align} \label{direct}
    \begin{aligned}
      \expect{\wt{\eta}_{j,j-2}} &\le \expect{\min \left(\frac{C\vep}{\min(\zeta_j, 2\pi-\beta_j - \zeta_j)} \nu_j, 1\right)} \\
    &\le C \sum_{n=1}^\infty \frac{(n+1)}{e^{2\pi n}}  \int_0^{2\pi} \min \left(\frac{\vep}{x}, 1 \right) d x\\
    &\le C \vep \abs{\log \vep}.
    \end{aligned}
  \end{align}

  Now we turn to $\wt{\eta}_{j,k}$ for $k < j-2$. Again, we begin by conditioning on $Y_{j-1}, Y_k, \xi_j$ and using Lemma \ref{lem:geometric} to bound
  \begin{align*}
    \condexpect{\wt{\eta}_{k,j}}{Y_{j-1},Y_k, \xi_{j}} \le  \min\left(\frac{C\vep}{\abs{Y_{j-1} - Y_{k}}} (\nu_j+1), 1\right).
  \end{align*}
  In the notation of \eqref{limit comps-eps} we can write
  \begin{align*}
    \condexpect{\wt{\eta}_{k,j}}{Y_{j-2},Y_k, \xi_{j}} \le  \int_0^\infty \int_{0}^{2\pi} \min\left(\frac{C\vep}{\abs{Y_{j-2} + \vep \wh{y}_{j-1} + \wt{y}_{j-1} - Y_{k}}} (\nu_j+1), 1\right)e^{-\xi_{j-1}} \mu_{\alpha}(d\alpha_{j-1})   d  \xi_{j-1}. 
  \end{align*}
  Note that both the probability densities for $\alpha_{j-1}$ and $\zeta_{j-1}$ are bounded from above. Therefore we can upper bound the above conditional expectation by
  \begin{align*}
    \condexpect{\wt{\eta}_{k,j}}{Y_{j-2},Y_k, \xi_{j}} &\le  \max_{\wh{y} \in B_1(0)} \left(\int_0^{2} \int_{0}^{2\pi} \min\left(\frac{C\vep}{\abs{Y_{j-2} - Y_k + \vep \wh{y} + v e^{i\alpha} }} (\nu_j+1), 1\right)  d\alpha d v  \right),\\
    &\le  \int_0^{2} \min\left(\frac{C\vep}{v} (\nu_j+1), 1\right)  d v, \\
    &\le C \vep \abs{\log \vep }(\nu_j+1),
  \end{align*}
  where $B_1(0)$ is the unit ball centered at $0$. We can uniformly bound this by 
  \begin{align*}
    \condexpect{\wt{\eta}_{k,j}}{Y_{j-2},Y_k, \xi_{j}} \le C \vep \abs{\log \vep }(\nu_j+1).
  \end{align*}
  From here, the tower law then implies
  \begin{align}\label{kj bound}
    \expect{\wt{\eta}_{k,j}} \le C \vep \abs{\log \vep}.
  \end{align}

  Putting \eqref{kj bound} and \eqref{direct} together, applying the union bound and the tightness of $\gamma$, we conclude \eqref{prob mismatch}.

\end{proof}

\subsection{Proof of Theorem \ref{thm:XY}}
\label{ss: Proof of Theorem thm: XY}

The rest of the proof follows {\cite[Section 7]{lutsko-toth_19}} fairly exactly. Let $\varpi_n$, $n \ge 1$ be an i.i.d sequence of packs. For each pack we denote $((Y_n(t)) : 0 \le t \le \theta_n)_{n \ge 1}$ the associated leg. Let $\{ t \mapsto Y(t)\}$ denote the full Markovian process decomposed into legs. Furthermore we define two discrete stopping times: 
\begin{align*}
  &\rho_1 := \min \{n \ge 1:  \eta_j^n =1 \mbox{ for some } j \in \{1,\dots, \gamma_n \}\}\\
  &\rho_2 := \min \{n \ge 1: \max \{\wt{W}_n, \wh{W}_n\} = 1 \}
\end{align*}
where $\eta_j^n$ is the same as the indicators in \eqref{mismatches} for the process $\{t \mapsto Y_n(t)\}$ (i.e the indicator that during the $j^{th}$ step of the $n^{th}$ leg there is a mismatch with another path segment in the same leg). Therefore, $\rho_1$ is the time of the first mismatch within one leg, and $\rho_2$ is the first time two legs intersect to give a mismatch.

With that the stopping time defining $X$ is given by $\rho = \min\{\rho_1,\rho_2\}$. And by construction, $X(t) = Y(t)$ for all $t \le \Theta_{\rho-1}$. Using Proposition \ref{prop:inter} and Lemma \ref{lem:mismatch} we have
\begin{align*}
  \begin{aligned}
    \probab{\Theta_{\rho-1} < T } &\le \probab{\rho_1 \le 2 \expect{\theta}^{-1}T} +  \probab{\rho_2 \le 2 \expect{\theta}^{-1}T} + \probab{ \sum_{j=1}^{2\expect{\theta}^{-1}< T}\theta_j < T}\\
    &\le C\vep \abs{\log{\vep}}T + Ce^{-cT},
  \end{aligned}
\end{align*}
for two constants $C<\infty$ and $c>0$. Thus
\begin{align} \label{Theta bound}
  \lim_{\vep\to 0} \probab{\Theta_{\rho-1} < T } =0,
\end{align}
for $T = o \left( \vep^{-1} \abs{\log{\vep}}^{-1}\right)$. From here Theorem \ref{thm:XY} follows immediately.

\qed

\section*{Acknowledgements}

The work of BT was supported by EPSRC (UK) Fellowship EP/P003656/1 and by NKFI (HU) K-129170.

\vskip2cm

\hbox{

\vbox{\hsize=7cm\noindent
{\sc CL's address:}
\\
Mathematics Department
\\
University of Houston
\\
Houston, 77004
\\
USA
\\
{\tt clutsko@uh.edu}
}

\vbox{\hsize=7cm\noindent
{\sc BT's address:}
\\
School of Mathematics
\\
University of Bristol
\\
Bristol, BS8 1TW
\\
United Kingdom
\\
{\tt balint.toth@bristol.ac.uk}
}

}

\end{document}